\documentclass[11p, authoryear]{elsarticle}
\usepackage{amsmath, amsfonts, amssymb, amsthm, eucal, enumerate, setspace}
\usepackage{natbib}
\newtheorem{theorem}{Theorem}
\newtheorem{lemma}{Lemma}
\newtheorem{definition}{Definition}
\newtheorem{proposition}{Proposition}
\newdefinition{rmk}{Remark}

\def\bei{\begin{itemize}}
\def\eei{\end{itemize}}
\def\beq{\begin{equation}}
\def\eeq{\end{equation}}
\def\bed{\begin{definition}}
\def\eed{\end{definition}}
\def\bet{\begin{theorem}}
\def\eet{\end{theorem}}
\def\bel{\begin{lemma}}
\def\eel{\end{lemma}}
\def\bep{\begin{problem}}
\def\eep{\end{problem}}
\def\bec{\begin{corollary}}
\def\eec{\end{corollary}}
\def\ack{\noindent{\bf{Acknowledgement.$\,\,$}}}
\newtheorem{corollary}{Corollary}
\def\E{{\ensuremath{\mathbb {E}}}}
\def\R{\ensuremath{\mathbb {R}}}

\def\Pp{\ensuremath{\mathbb {P}}}
\def\eE{\mathcal{E}}
\def\eM{\mathcal{M}}
\def\eB{\mathcal{B}}

\def\tow{{\stackrel{w}{\longrightarrow}}}

\def\({\left(}
\def\){\right)}

\def\al{\alpha}

\def\be{\beta}

\def\si{\sigma}

\def\la{\lambda}
\def\iy{\infty}

\def\Mix{{\operatorname{Mix}}}

\def\ga{\gamma}

\def\cB{{\mathfrak B}}

\def\mC2#1{{\mathbf C}^{(2)}(#1)}
\def\mC#1#2{{\mathbf C}^{(#1)}(#2)}

\def\cF{{\mathfrak F}}
\def\cM{{\mathfrak M}}

\def\bydef{:=}

\def\t{\tau}

\def\d{\,\mathrm{d}}

\def\text#1{\mbox{ #1 }}

\def\i0t{\int_0^t}

\def\gh{\hat g}
\def\beq{\begin{equation}}
\def\eeq{\end{equation}}

\usepackage{graphicx}
\graphicspath{%
    {converted_graphics/}
    {/}
}
\begin{document}
\title{The Generalized Shiryaev's Problem \\ and
\\ Skorohod Embedding\tnoteref{trf}}
\author[sj]{Sebastian Jaimungal\fnref{t2}}
\ead{sebastian.jaimungal@utoronto.ca}
\author[ak]{Alexander Kreinin\corref{cor1}}
\ead{alex.kreinin@algorithmics.com}
\author[sj]{Angel Valov}
\ead{valov@utoronto.ca}
\address[sj]{Department of Statistics, University of Toronto;
100 St. George Street, Toronto,\\ ON, , Canada.}
\address[ak]{Quantitative Research, Algorithmics Inc.; 185 Spadina Ave,
Toronto, ON, M5T 2C6, Canada.}

\cortext[cor1]{Corresponding author}
\tnotetext[trf]{This research was
supported in part by the Natural Sciences and Engineering
Research Council (NSERC) of Canada.}
\fntext[t2]{URL: http://www.stat.toronto.edu/~}

\begin{abstract}
In this paper we consider a connection between the famous Skorohod embedding
problem and the Shiryaev inverse problem for the  first hitting time distribution of a  Brownian motion: given a probability distribution, $F$, find a boundary such that the first hitting time
distribution is $F$. By randomizing the initial state of  the process  we show that the inverse problem becomes analytically tractable. The  randomization of the initial state allows us to significantly extend the class of target distributions in the case of a linear boundary and moreover allows us to establish connection with the Skorohod embedding problem.
\end{abstract}

%
\maketitle

\section{Introduction}
Let $X_t$ be an arbitrary process with the left-continuous sample paths
and let $b(t)$ be  a continuous absorbing boundary
satisfying the condition $X_0 \ge b(0)$.  The random variable
\[
    \t = \begin{cases} \inf\{t \ge 0:\, X_t < b(t)\},
                  &\text{if there exists $t$ such that $X_t<b(t)$, }\\
               \infty, &\text{otherwise,}\end{cases}
\]
is called the first hitting time for the process $X_t$.
The problem
\begin{eqnarray}
& &\text{\it Given a process, $X_t$, and the boundary,
$b(t)$, } \label{probl_1} \\
& &\text{\it find the distribution
$ F_\t(t)=\Pp\(\t\le t\)$.} \nonumber
\end{eqnarray}
is the starting point of a very  rich research area in the theory of stochastic processes.
If $X_t$ is a diffusion process, the problem of finding the distribution
of $\t$ is a classical one. The first papers on the problem were published by
P.~Levy, A.~Khintchine and A.~Kolmogorov in the 1920s. These results are discussed in the
monograph\footnote{Khintchine called (\ref{probl_1})
the second problem of diffusion.}  \citet{khint}.
Khintchine  gave a complete solution of the problem for sufficiently smooth boundaries.
This solution was expressed in
terms of a boundary value problem for the associated partial differential operator (infinitesimal generator).
Since then many books and research
papers were published in this area. The monograph
\citet{Lerch} summarizes known analytical results obtained by the mid-1980s.
In \citet{Karatz} the link between analytical methods and the martingale
approach is considered, the paper \citet{Durb} discusses computational aspects of the problem,
while the Taylor expansions of the probability distribution of $\t$ are
considered in \citet{HobWil}.

The following,  inverse  to~(\ref{probl_1}), problem was proposed
by Albert Shiryaev in the mid-$70$'s in his Banach Center lectures:
\begin{eqnarray}
& &\text{\it Given a process, $X_t$, and a distribution,
$F(t)$, } \label{probl_AS} \\
& &\text{\it find a boundary, $b(t)$, such that
$ \Pp\(\t\le t\)=F(t)$.} \nonumber
\end{eqnarray}
 Problem~(\ref{probl_AS}), to the best of our knowledge, was  proposed by
A. Shiryaev in the case when
$X_t$ is a Brownian motion and $F(t)$ is an exponential distribution.
If $X_t$ is a Brownian motion with a random initial value, i.e., the process
$X_t=\xi+W_t$ is a Brownian motion starting at a random point $\xi$,
we will refer to Problem~(\ref{probl_AS}) as a generalized Shiryaev's problem (GSP).

The third problem  considered in this paper is the  Skorohod embedding problem:
\begin{eqnarray}
& &\text{\it Given a probability measure, $\mu$,
with a finite second moment,  }\label{probl_Skor} \\
& &\text{\it and
a Wiener process, $W_t$, find an integrable stopping time, $\t_\ast$, }\nonumber \\
& &\text{\it such  that the distribution of $W_{\t_\ast}$ is $\mu$.}\nonumber
\end{eqnarray}
This problem has a long list of references. Here we mention only
the   original paper \cite{Skoroh} and a survey by \cite{obloj}
where many important applications and various solutions of the Skorohod embedding problem can be found.

It is intuitively appealing that there must be a connection between the
problems (\ref{probl_AS}) and
(\ref{probl_Skor}). Indeed if a monotone continuous function $b_\ast(\cdot)$ solves
Problem~(\ref{probl_AS}) for a Wiener process $W_t$, and the
distribution function $F(t)$,
then the  first hitting time, $\t_\ast$, is a solution of the Skorohod problem with
$\Pp\( W_{\t_\ast}\le t\)= F\( b_\ast^{-1}(t)\)$. The only problem with the latter condition
is that $b_\ast(\cdot)$ is usually unknown in analytical form.
Moreover, the boundary might not belong to the class of monotone functions.

Models using the first hitting time distribution find extensive application in
the areas of portfolio credit risk modeling (see \citet{IsKr}, \citet{IsKrRos},
\cite{Novikov08}) and pricing of credit
derivatives (\citet{Avel}, \citet{HW}). In these contexts the process represents the so-called
{\it distance to default} of an obligor (see \citet{Avel}),
while the first hitting time represents a default event.
The boundary therefore acts a barrier separating the healthy states of the obligor from the default state.
For this reason, the boundary, $b(t)$, is often called the {\it default boundary} in
the applied literature on credit risk modeling. In particular, an interesting
model of default events with a randomized boundary was proposed in \cite{Novikov08}.

The inverse problem (\ref{probl_AS}) was considered in
\citet{IsKrRos} when the process $X_t$ is a Brownian random walk\footnote{A Brownian random walk
is a discrete time process with Gaussian increments and variance proportional to the time step.}.
A detailed analysis of the inverse problem in the discrete time
setting  is given in \citet{IsKr}
as well as a Monte Carlo based solution.
Their approach is applicable to a much more general class of processes $X_t$,
not just Brownian random walks, and is computationally simple to implement.

Existence of the solution to the continuous-time inverse problem (\ref{probl_AS}) is analyzed in
\citet{Saun}. In \citet{Pesk}, an integral equation for the boundary is derived
when $X_t$ is a Brownian motion. A general analysis of the integral equations for the boundary
 and existence and uniqueness theorems are considered in \citet{JKV_integr}.
The randomized inverse
problem is considered in \citet{JKV_rand} and in the preliminary publication \citet{JKW}.

The Shiryaev  problem is notoriously  difficult and
analytical solutions are known only in a few cases only (see \cite{Shepp67}, \cite{Salm88},
\cite{Lerch}, \cite{PeskShir}, \cite{Breiman66}, \cite{Alili05},
\cite{Novikov81}, \cite{Pesk}). However, existence of the solution to the problem has been proven for an arbitrary target distribution, $F$, by \cite{Dudley} and by \cite{Anul}.
Unfortunately, an analog of this existence theorem
in its most general form cannot be proven for the {\it the Generalized Shiryaev's problem} and in fact counter-examples do exist (see \cite{JKV_rand} and Proposition 2 below). Nonetheless, the randomized version does admit closed form solutions for a large class of distributions of the first hitting time, $F$, including the gamma distribution (see \cite{JKW} and more generally \cite{JKV_rand}) as well as a subset of one-sided stable distributions (see Corollary 1 below).

One of our main goals is to establish a connection between the Generalized Shiryaev's problem and the Skorohod embedding problem. The remainder of this paper is organized as follows. In Section~\ref{sec_randomiz} we proved a new short derivation for the Laplace transform of the distribution of the first hitting time and the
distribution of the initial random position $\xi$ of the process $X_t$. When the boundary is linear, we
find solutions for a class of mixtures of
gamma distributions of the first hitting time and for a class of stable
distributions. In Section~\ref{sec_str_solut} we analyze the structure of the solutions
to the Shiryaev's problem and introduce the minimal solution which possess a very elegant
structure. The randomization of the initial state of the process allows us to stretch the boundary and transform it into a straight line. In this case we have a simple relation between the distributions of $\t$ and $X_\t$. This observation allows us to connect the Skorohod embedding problem to the Generalized Shiryaev problem as shown in Section~\ref{sec_conn_Skorohod}.

This paper is self-contained; it represents an extended version
of the talk given at the $5^{th}$ Bachellier colloquium, January 2011,  Metabief.

\ack{We are very grateful to Yuri Kabanov and Tom Sailsbury for
pointing out the appealing similarity
of the Problems under consideration and asking about the connection
between the inverse Problem~(\ref{probl_AS}) and the Skorohod embedding problem.
We would also like to thank Lane Hughston and Mark Davis for the interesting comments on earlier stages of this work. The authors are indebted  to Albert Shiryaev, Alexander Novikov,
Raphael Douady, Isaac Sonin and Nizar Touzi for the  interesting discussions during the $5^{th}$ Bachellier colloquium. }

\section{Randomization of the initial state}\label{sec_randomiz}
\subsection{General equation}
Let $\(\Omega, \cB, \cF_t, \Pp\)$ be a filtered probability space and let
$\Omega= {\mathbf{C}}\([0, \iy)\)\times\R^1$
 be the Cartesian product of the space of continuous functions on the positive semi-axis and $\R^1$.
Consider a random variable $\xi$ on this probability space.
We assume that $\{\cF_t\}=\{\cF_t^W\} \bigvee \eB(\R^1)$ is an augmented filtration
where  $\cF_t^W$ is the natural filtration, generated by the standard
Wiener process $W_t$ and
assume that $\cB = \bigcup_{t\ge 0}\limits \cF_t$.
We also  assume that  the random variable $\xi$ is independent
of the process $W_t$.

Let $X_t=\xi+W_t$ and let $\Pp\(0\le \xi < +\iy\)=1$.
Consider now the linear  boundary,  $b(t) = k \,t $, $k\ge 0$.
Denote the first hitting time of the process $X_t$ to $b(t)$  by $\t^{(k)}$, i.e.,
$$ \t^{(k)}=\inf\{t: t>0, \,\,\xi+W_t<k\,t\}. $$
The Generalized Shiryaev's Problem (GSP) for the process $X_t$ is formulated as follows:
\beq
\begin{split}
\text{\it Given  a distribution $F(t)$,
    find $\xi$, } \\
\text{\it such that $\Pp\(\t^{(k)}\le t\)=F(t)$. }
\end{split}
\label{probl_Id}
\eeq
The triplet $(\xi, \t^{(k)}, k)$ is called a solution of the Generalized Shiryaev's
Problem.

We are interested in the solutions to the GSP and their properties.
In particular, we are interested in the solution of Problem~(\ref{probl_Id})
when $F(t)$ belongs to the class of Gamma distributions, $F\in\Gamma_{\la, \gamma}$,
with the probability density function
$$ p_\gamma(t) = \la \cdot\frac{(\la \,t)^{\gamma -1}}{\Gamma(\gamma)}\cdot \exp(-\la \,t),
\qquad\gamma >0, \,\,\la>0. $$
\begin{rmk}
The case $\gamma=1$ corresponds to the original formulation of  Shiryaev's problem with a randomized initial point. In this case
 the distribution $F(t)=1-e^{-\la t}, \,\, t\ge 0$.
 \end{rmk}

\noindent To solve the problem, we will derive a connection between the Laplace transforms of the target distribution $F(t)$ and the initial starting point $\xi$. To this end, denote
\beq
    \hat f(s)= \E[e^{-s\t}]= \int_0^\iy e^{-st}\d F(t), \quad s\in \R_+,
\label{eq_def_fh}
\eeq
where, as usual, $\R_+=\{ s: s\ge 0\}$, is the set of non-negative real numbers,
and let $\hat g(s) = \E\bigl[ e^{-s\xi}\bigr]$.
The function $\hat f(s)$ on $[0, \iy)$, is completely monotone (see \cite{Feller}),
$\hat f\in \cM$,
where
$$
    \cM=\biggl\{ V(s): (-1)^n \frac{\d^n V(s)}{\d s^n}\ge 0,\,\, n=0, 1, \dots,
                       \,\, s\in\R_+\biggr\}.
$$
The class of completely monotone functions form an algebra: the sum and the product
of completely monotone functions belong to $\cM$. According to
the classical Bernstein's Theorem (see \cite{Feller}),
if $\hat f(s)\in\cM$ and $\hat f(0)=1$, then
there exists a cumulative distribution function, $F(t)$, satisfying (\ref{eq_def_fh}).

The following statement for   completely monotone functions
is very well known (see \cite{Feller}, Criterion $2$):
\begin{proposition}\label{prop_cm}
If $G(s)$ is a completely monotone function, $u(s)\ge 0$ and the first derivative
$u^\prime\in\cM$ then $G(u(s))\in\cM$.
\end{proposition}
\noindent
Let us now derive the main equation for the Laplace transforms of the distributions of
$\t$ and $\xi$.
\begin{theorem}\label{theor_LT_boundary}
If a random variable $\xi$ is a solution to the Problem~{\rm{(\ref{probl_Id})}}
then the function $\hat g(s)$ satisfies the equation
\beq
\gh(s) =\hat f\(sk + s^2/2 \).
\label{eq_g_hat}
\eeq
\end{theorem}
\begin{proof}
The process $X_t$ is a martingale with respect to the  filtration $\{\cF_t\}_{t\ge 0}$.
Consider the exponential martingale $M_t = \exp\(-sX_t - s^2 t/2 \)$, $(s>0)$.
It is not difficult to show that if $\xi$ is finite almost surely then
the exponential martingale is bounded for all $t\le\t$. 
Therefore, by the optional stopping theorem,
$$ \E[M_\t] = \E[M_0].$$
Moreover, the expected value
$$ \E[M_0] = \E[e^{-s\xi}]. $$
Furthermore, he have $X_\t=k\,\t$ and
$$M_\t = \exp\(-\t\cdot (sk + s^2/2)\). $$
Therefore $\E[M_\t]=\hat f(ks+s^2/2)$.
Finally, we obtain Equation~(\ref{eq_g_hat}).
\end{proof}

\subsection{Gamma-distributed first hitting time}
In the case $f\in\Gamma_{\la, \gamma}$, the random variable $\xi$
solving Problem~(\ref{probl_Id})
admits the following simple probabilistic interpretation.
\begin{theorem} \label{theo_a_gamma}
Let
$ f(t) = p_\gamma(t)$,
for some $\gamma>0$.  
Suppose that $2\la \le k^2$.
Then $\hat g(s)$ satisfies
\beq
\hat g(s)= \frac{(k+\sqrt{k^2-2\la})^\gamma}{\(s+k +\sqrt{k^2-2\la} \)^\gamma}
\cdot  \frac{(k-\sqrt{k^2-2\la})^\gamma}{\(s+k -\sqrt{k^2-2\la} \)^\gamma}.
\label{eq_sol_gam}
\eeq
and the random variable
\beq
 \xi=\xi_1 + \xi_2,
\label{eq_rep_xi}
\eeq
where $\xi_1$ and $\xi_2$ are independent gamma-distributed random variables
with a common shape parameter $\gamma$:
$\xi_1\sim\Gamma(\gamma, \la_1)$ and
$\xi_2\sim\Gamma(\gamma, \la_2)$, with
\beq
\la_1=k-\sqrt{k^2-2\la}, \qquad \la_2=k +\sqrt{k^2-2\la}.
\label{eq_la12}
\eeq
\end{theorem}
\begin{proof}
 We have
$$ \hat f(s) = \frac{\la^\gamma}{(\la+s)^\gamma}. $$
From (\ref{eq_g_hat}) we obtain
$$ \hat g(s) = \(  \frac{2\la }{s^2 + 2ks + 2\la }  \)^\gamma. $$
If $k^2\ge 2\la$ the quadratic equation
$$ s^2 + 2k s + 2\la=0$$
has real roots $-\la_1$ and $-\la_2$.
Then, taking into account that $\la_1 \la_2 = 2\la$, we find
$$
    \hat g(s) = \( \frac{\la_1}{\la_1+s}\)^\gamma \cdot
               \( \frac{\la_2}{\la_2+s}\)^\gamma .
$$
The additive representation (\ref{eq_rep_xi}) for $\xi$ follows immediately
from the latter equation. \end{proof}
\begin{rmk}
 It follows from Theorem~\ref{theo_a_gamma} that if the GSP
 has a solution for $k=k_\ast$, given $f(t)\in \Gamma_{\la, \gamma}$, then
it also has a solution for any $k> k_\ast$. We will generalize this property for
a larger class of distributions in the next section.
\end{rmk}

 Theorem~\ref{theo_a_gamma} can also be generalized in the following direction.
Consider a class of random variables $\eM_\Gamma$ which are mixtures of gamma-distributed
random variables with respect to shape $\gamma$ and scale $\la$ parameters. Write,
$$\t=\Mix\(\Gamma_{\la, \ga}, \mu\), $$ where $\mu$ is a mixing measure having a
support, $(\la, \gamma)\in S=[0, \la_\ast]\times [0,\iy)$, $\mu\(S\)=1$, and
$\la_\ast=k^2/2$. Then
$$\hat f(s) \bydef \E\bigl[ e^{-s\t}\bigr]=
\int_S \frac{\la^\gamma}{(\la+s)^\gamma}\d \mu(\la, \gamma). $$
In this case
  \beq \hat g(s) = \int_S \frac{\la_1^\gamma(\la)}{(\la_1(\la)+s)^\gamma}
  \frac{\la_2^\gamma(\la)}{(\la_2(\la)+s)^\gamma} \d \mu(\la, \gamma), \label{eq_mix_gam_sol}
\eeq
where $\la_1(\la)$ and $\la_2(\la)$ satisfy (\ref{eq_la12}).
Therefore,  we see that the random variable $\xi$ which leads to the hitting time distribution $\Mix\(\Gamma_{\la,\ga},\mu\)$ is in fact a  mixture of the sums of independent gamma-distributed  random variables.

This generalization is interesting because it is known that the class of finite mixtures of gamma-distributed random variables
$\eta\in\eM_\Gamma$ form a dense set
in the space of random variables. Moreover, an arbitrary random variable $\eta_*$ can be obtained as a weak limit as the number of gamma distributions $n\to\iy$,
$$ \eta_n \tow\;\eta_*, \qquad \eta_n\in\eM_\Gamma . $$
Unfortunately, this sequence of approximations can not be used to extend  the solution
(\ref{eq_mix_gam_sol}) through weak convergence because the parameter $\la$ has a bounded
support. In particular, we cannot approximate constant almost surely random variables using
the random variables, $\eta_n\in \eM_\Gamma$.

\subsection{Stable distribution of the first hitting time}

So far, we have considered the case when $\t$ has finite mean, $\E\,\t<+\iy$.
However, it is also interesting to analyze the case of one-sided
stable distributions with infinite expected value. For this case, we take the slope coefficient $k=0$.
Then, if one takes the random variable $\t$ such that
$$\E\bigl[ e^{-s\t}\bigr] =e^{- c\,s^\al}, \quad c>0, \quad 0<\al<1,  $$
then it follows from (\ref{eq_g_hat}) that
$$ \hat g(s) = \exp\(-c_1 s^{\be}\), $$
where $\be=2\al$ and $c_1 = \frac{c}{2^\al}$. Thus, we obtain
\begin{corollary}\label{cor_stab_tau}
If $\t$ has a stable distribution with the parameter $\al$, $0<\al<1/2$,
then $\xi$ has a stable distribution with the parameter $2\al$.
\end{corollary}

\section{Structure of the solutions to Shiryaev's problem}\label{sec_str_solut}
\subsection{General results}
Theorem~\ref{theo_a_gamma} describes the solutions to the Shiryaev's problem
when
the first hitting time $\t$ is  gamma-distributed. In this case, it can be shown that the condition
$ k\ge k_\ast=\sqrt{2\,\la} $
is necessary and sufficient for existence of the random variable $\xi$ (see \cite{JKW} and \cite{JKV_rand}).
Thus if the problem has a solution for the slope $k_\ast$ then it has a solution
for any $k>k_\ast$. The following statement generalizes this result.
\begin{theorem}
\label{theo_str_sol_gen} Given $\t$,
suppose there exist $k_\ast>0$ and $\xi_\ast=\xi(k_\ast)$  such that
Equation~\rm{(\ref{eq_g_hat})} is satisfied:
$$ \hat g_\ast(s)\bydef  \E\bigl[ e^{-s\xi_\ast}\bigr]
= \hat f\(k_\ast s + s^2/2\). $$
Then for any $k>k_\ast$ there exists a random variable $\xi(k)$ such that
\beq
\E\bigl[ e^{-s\xi(k)}\bigr] = \hat f\( ks+s^2/2\), \quad s\ge 0.
\label{eq_rel_xi_k}
\eeq
\end{theorem}
\begin{proof} The function $\hat g(s)= \hat f\( ks+s^2/2\)$
satisfies the relation $\hat g(0)=1$. Therefore, it is enough
to prove that $\hat g\in\cM$. Let $k>k_\ast$. Consider the function
$$ u(s)=\sqrt{s^2 +2ks + k_\ast^2 } - k_\ast. $$
Obviously,
$ u: \R_+ \to \R_+$,
and
$$ \frac{s^2}{2} +ks = \frac{u^2(s)}{2} + k_\ast u(s), \quad s\ge 0.$$
The first derivative of $u(s)$ is
$$ \frac{\d}{\d s} u(s)\equiv u^\prime(s) =
   (s+k)\cdot \(s^2+2ks+k_\ast^2\)^{-\frac12}.$$
Let us prove that the function $u^\prime(s)\in\cM$. We have
\beq
  u^\prime(s) \cdot\(u(s)+k_\ast\) = s+k.
\label{eq_fdu}
\eeq
Then if $k>k_\ast$, we derive from (\ref{eq_fdu}) the inequality
\beq
    u^\prime(s) > 1.
\label{eq_ineq_fd}
\eeq
Denote $n^{th}$ derivative of the function $u(s)$ by $ u^{(n)}(s)$. $(n=1, 2, \dots)$.
Differentiating Equation~(\ref{eq_fdu}) we find
\beq
    u^{(2)}(s)\cdot\(u(s) + k_\ast\) + \( u^\prime(s) \)^2=1.
\label{eq_u2}
\eeq
The latter Equation and (\ref{eq_ineq_fd}) imply the inequality
$$ u^{(2)}(s)<0, \quad s>0.$$
Denote $m=\lceil n/2\rceil$.
From Equation~(\ref{eq_u2}) we derive
\beq
u^{(n)}(s) \cdot \(u(s) + k_\ast\) = -\sum_{j=1}^{m} C_{n, j} u^{(j)}(s)\cdot u^{(n-j)}(s),
\label{eq_nder_u}
\eeq
The coefficients $C_{n,j}$ satisfy the relations
\begin{eqnarray*}
  C_{n+1, 1} &=& C_{n,1}+1, \quad n=2, 3, \dots,  \\
  C_{n+1, j} &=& C_{n, j} + C_{n, j-1}, \quad n=2,3, \dots, j = 2, \dots,  m-1.
\end{eqnarray*}
Taking into account that $C_{2,1}=1$, we find
\beq
    C_{n, j}= {n\choose j}, \quad j=1, 2, \dots, m-1. \label{eq_Cnj}
\eeq
The coefficient $C_{n, m}$ satisfies the relation
\beq
C_{n, m}=\begin{cases} {n\choose m}, &\text{if $n$ is odd, }\\
               \frac12 {n\choose m}, &\text{otherwise,}\end{cases}
\label{eq_Cnm}
\eeq
Thus, $C_{n, j}>0$ in Equation~(\ref{eq_nder_u}). Then from (\ref{eq_nder_u})
we prove by induction that
$$(-1)^{n+1} u^{(n)}(s)>0\quad\text{ for all $n=1, 2, \dots$.} $$
 Therefore $u^\prime(s)\in\cM$.
Then we have
$$ \hat g(s) =\hat f\( \frac{u^2(s)}{2} + k_\ast u(s)\) =
\hat g_\ast\(u(s) \). $$
Since $\hat g_\ast\in\cM$ and $u^\prime(s)\in\cM$ we derive from
Proposition~\ref{prop_cm} that
$ \hat g(s)\in\cM$, as was to be proved.
\end{proof}

The next proposition demonstrates that there exist parameters of the GSP 
such that the problem does not have a solution.
\begin{proposition} \label{lem_nex_sol} Suppose $\E[\t^2]<\iy$. Then
the Problem~(\ref{probl_Id}) does not have a solution
if $k<\sqrt{\E[\t]}/\si(\t)$.
\end{proposition}
\begin{proof}
Indeed, from Equation~\rm{(\ref{eq_g_hat})} we find
$$ \frac{\d \hat g(s)}{\d s} =  \frac{\d \hat f(ks+s^2/2)}{\d s} \cdot (k+s), $$
and
$$ \frac{\d^2 \hat g(s)}{\d s^2} =  \frac{\d^2 \hat f(ks+s^2/2)}{\d s^2} \cdot (k+s)^2
+  \frac{\d \hat f(ks+s^2/2)}{\d s}\,\,.  $$
Substituting $s=0$ into these equations, we find
\begin{eqnarray*}
     \E[\xi] &=& k\cdot \E[\t],  \\
   \E[\xi^2] &=& k^2 \E[\t^2] - \E[\t].
\end{eqnarray*}
The equations for the first  two moments of $\xi$ imply
$$ k^2 \si^2(\t) = \E[\t] + \si^2(\xi) \ge \E[\t]. $$
Finally, we obtain
$$ k\ge\frac{ \sqrt{\E[\t]}}{\si(\t) }. $$
Proposition~\ref{lem_nex_sol} is thus proved.
\end{proof}
\begin{rmk}Proposition~\ref{lem_nex_sol} provides only a necessary condition
for existence of a solution to Problem~(\ref{probl_Id}).
In particular, if $\t\sim \Gamma(\gamma, \la)$,
then the necessary and sufficient condition for existence is $k\ge \sqrt{2\la}$ and
the minimal solution corresponds to $k_\ast=\sqrt{2\,\la}$.
\end{rmk}
\begin{rmk}
If $\t$ is a constant with probability $1$, Equation~(\ref{eq_g_hat})
does not have a solution.
The GSP can not be solved using
our randomization approach in this case.
\end{rmk}

Theorem~\ref{theo_str_sol_gen} and Proposition~\ref{lem_nex_sol} imply
that the admissible set for the coefficient $k$ is either the semi-infinite interval
$[k_\ast, \iy)$ or an empty set.
If $\t$ has finite first two moments and the random variable $\xi$ exists for some
$k_\ast>0$ then for each $k>k_\ast$ there exists a random variable $\xi(k)$ solving
Problem~(\ref{probl_Id}). In this manner, we can obtain a family of solutions $\Xi=\{ \xi(k)\}$
parameterized by the slope coefficient $k$.
Furthermore, the variance of the random variable $\xi(k)$ is a monotone
function of $k$ and the one with minimal variance is what we call the
{\it minimal solution} to Problem~(\ref{probl_Id}). In particular, if $\t$ has an exponential distribution with parameter $\la$, then
the random variable $\xi_\ast$ corresponds to the minimal solution has slope
$k_\ast=\sqrt{2\la}$. In this case  the
random variable $\xi(k_\ast)$ is Erlang distributed with order two.

 \subsection{Esscher families}

Consider a non-negative random variable $\eta$ and
denote $\hat f_\eta(s)=\E[e^{-s\eta}], \,\,s\ge 0$ its Laplace transform.
It is convenient to introduce a family of random variables,
$$ \eE(\eta) = \Bigl\{ \eta_a:
      \E[e^{-s\eta_a}] = \hat f_\eta(a+s)/\hat f_\eta(a), \quad a>0 \Bigr\},
$$
which we call the {\it Esscher family generated by the random variable $\eta$}.
As before, let $\t$ be the first hitting time and $\hat f(s)=\E[e^{-s\t}]$ be its
Laplace transform. Suppose the generalized Shiryaev's Problem~(\ref{probl_Id})
has a solution $\xi=\xi(k_0)$ for some fixed slope $k_0$. Consider now two
Esscher families $\eE(\t)$ and $\eE(\xi)$ generated by $\t$ and $\xi$, respectively.
\begin{proposition} For each $k>k_0$ there exist $\al=\al(k)$ and $\be=\be(k)$
such that $\t_{\al}\in\eE(\t)$ and
$\xi_{\beta}\in \eE(\xi)$ and $(\t_{\al}, \xi_{\be}, k)$ is a solution to the
Generalized Shiryaev's problem.
\end{proposition}
\begin{proof} Denote $\Delta =k-k_0$, $(\Delta >0)$, and take
$\al=\Delta \cdot k_0 + \Delta^2/2$, $\be=\Delta$. Consider the random variable
$\xi_\be\in\eE(\xi)$. Then after simple transformations we find
\begin{equation*}
\gh_\be(s)=\frac{ \hat f\(\al + (k_0 + \Delta)\cdot s + s^2/2\)}{\hat f(\al)} 
=  \frac{ \hat f\(\al + k s + s^2/2\)}{\hat f(\al)} 
= \hat f_\al(ks+s^2).
\end{equation*}
Thus, the triplet $(\t_\al, \xi_\be, k)$ form a solution of the problem.
\end{proof}

\section{Connection to  Skorohod problem.}\label{sec_conn_Skorohod}

We are now in a position to construct a solution to the Skorohod problem through the solution to the Generalized Shiryaev's problem.
Suppose that the distribution of the
stopped process $X_\t$ has  a non-negative  support, $\Pp(X_\t\ge 0)=1$, and the continuous
distribution, $F_X(x)=\Pp(X_\t\le x)$, can be matched by a corresponding distribution of
the random variable $\xi$ for some $k>0$. Taking into account that $X_\t=k\,\t$ for the
linear boundary $b(t)=k\,t$, $(k>0)$, we immediately
obtain that $\Pp(\t\le t)=F_X(k\,t)$ and the Laplace transform of the density of $\t$ satisfies
$$\hat f_\t(s)=\hat f_X\(\frac{s}{k}\).$$
In this case
$$ \hat g(s)= \hat f_X\(\frac{s^2}{2k^2}+s\).$$
Similarly, the Skorohod problem can be solved for the distributions with $\Pp(X_\t<0)=1$.
In this case, the first hitting time is understood as
$$\t=\inf_{t>0}\{t: X_t> k\,t, \,\,\, k<0\}, $$
i.e., the first time the process touches the boundary from below. Clearly, the initial point $\xi$ must then have support on the negative real axis.

The solution to the general Skorohod problem -- without restrictions on the support
of the distribution of $X_\t$ --  can then be obtained as follows: Let us represent the random variable $X_\t$ as a mixture of the random variables $X_+>0$ and
$X_-<0$:
$$ X_\t = \begin{cases} X_{+}, & \text{ with probability $p_+=\Pp\(X_\t\ge 0\)$,} \\
      X_{-}, & \text{ with probability $p_-=1-p_+$.}
\end{cases}
$$
Suppose that the Skorohod problem for the random variables $X_+$ and $X_-$
can be solved by the boundaries
$b_+(t)=k_+\cdot t$ and $b_-(t)=k_-\cdot  t$ and the random variables $\xi_+$
and $\xi_-$, respectively (see Figure~\ref{fig:Skorohod_problem}).
Then we have the following
\begin{proposition}\label{sol_Skor}
The random variable
$$
    \xi=\begin{cases} \xi_+, & \text{ with probability $p_+$,} \\
      \xi_-,  & \text{ with probability $p_-$,} \end{cases}
$$
and the boundary $b(t)=b_+\cup b_-$
solve the Skorohod problem for the random variable $X_\t$.
\end{proposition}

\begin{figure}[tbp] 
  \centering
  \includegraphics[bb=71 240 562 544,width=5.07in,height=3.0in,keepaspectratio]{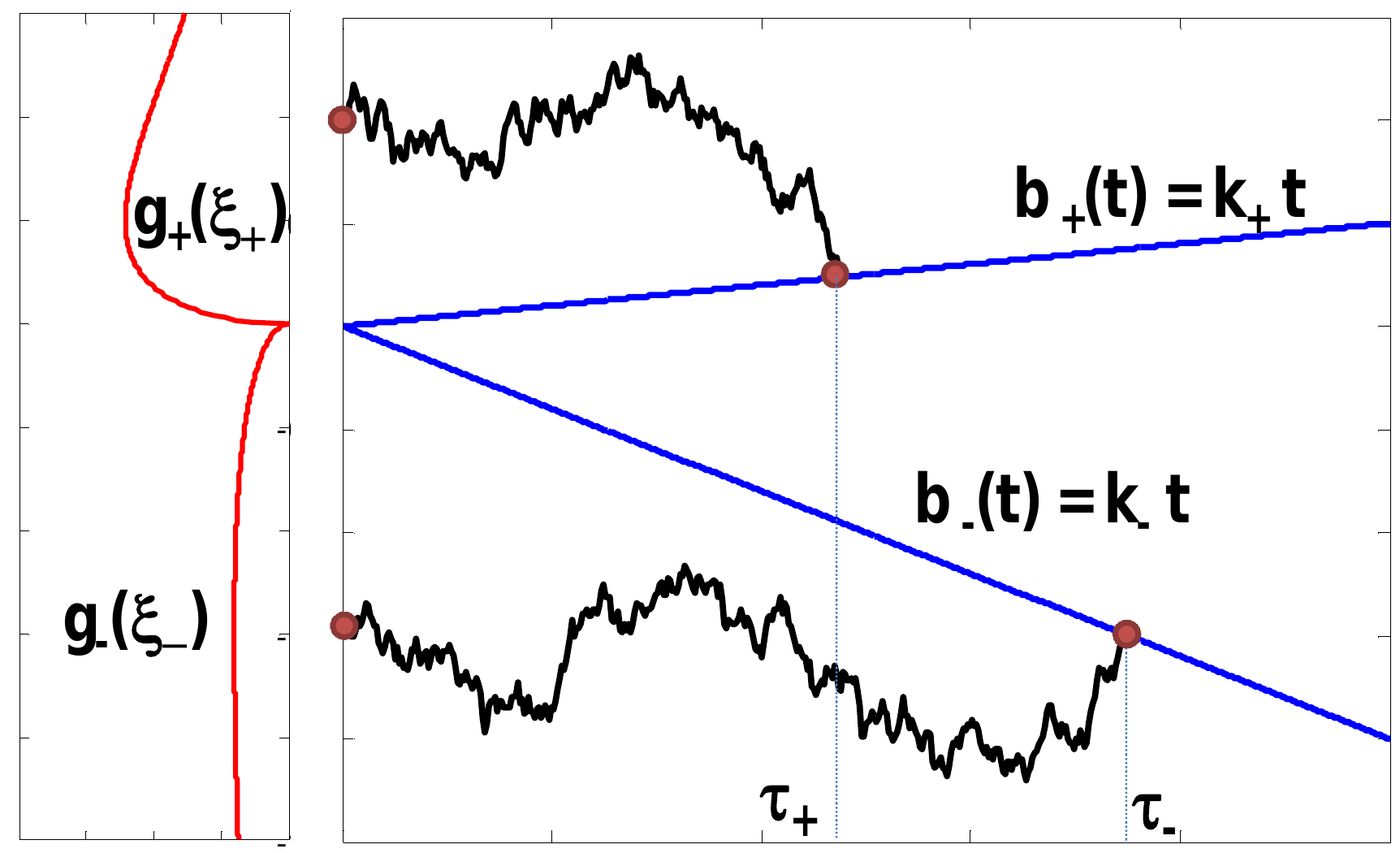}
  \caption{Randomization in the Skorohod Problem}
  \label{fig:Skorohod_problem}
\end{figure}

\begin{rmk}
In this manner, the Skorohod problem is solved by randomizing the starting point of the Brownian motion over the entire real line,
and searching for the first hitting time when the process enters the
wedge region described by the lines $k_+\,t$ and $k_-t$.
\end{rmk}

 \section{Conclusion}
We have demonstrated that randomization of the initial state of the process is a very powerful tool for solving and tying together the Generalized Shiryaev's Problem and the
Skorohod Embedding Problem. The randomization of the initial state of the process allows us to
reduce the problem to the linear boundary case and obtain closed-form solutions for several important classes of the distribution of the first hitting time.

This linearization of the boundary makes the relation between the Shiryaev's and Skorohod problems transparent.  Furthermore, it allows us to develop a new  solution to the Skorohod problem.

We close this paper with a brief discussion of the directions for future research.
In the present
paper we discussed the case of a scalar process $W_t$ whose distribution of the
first hitting time must be matched. One interesting open question is to prove existence of the
solution to the Generalized Shirayaev's problem when $\t$ has a continuous infinitely divisible distribution. Another interesting generalization that requires much attention is the Generalized Shirayaev's problem for a vector-valued process $W_t$. Such generalization will prove extremely useful in application settings as well as providing an interesting mathematical playground. 

\clearpage

\bibliographystyle{elsarticle-harv}
\bibliography{Skorohod_JKV}

\end{document}